\documentclass[12pt]{amsart}
\usepackage{amssymb,amsmath,amsfonts,epsfig,latexsym,texdraw}
\usepackage[all]{xy}
\newtheorem{theorem}{Theorem}[section]
\newtheorem{definition}[theorem]{Definition}
\newtheorem{proposition}[theorem]{Proposition}
\newtheorem{lemma}[theorem]{Lemma}

\theoremstyle{definition}

\newtheorem{example}[theorem]{Example}

\def\Z{\ensuremath{\mathbb{Z}}}
\def\Q{\ensuremath{\mathbb{Q}}}
\def\P{\ensuremath{\mathbb{P}}}

\def\C{\ensuremath{\mathbb{C}}}

\def\R{\ensuremath{\mathbb{R}}}
\def\K{\ensuremath{\mathbb{K}}}

\def\k{\ensuremath{\mathbf{k}}}
\def\U{\ensuremath{\mathbf{u}}}
\def\V{\ensuremath{\mathbf{v}}}
\def\x{\ensuremath{\mathbf{x}}}

\def\ca{\ensuremath{\mathcal{A}}}

\def\cd{\ensuremath{\mathcal{D}}}

\def\cm{\ensuremath{\mathfrak{m}}}
\def\ccr{\ensuremath{\mathcal{R}}}
\def\cn{\ensuremath{\mathcal{N}}}

\def\excise#1{}

\def\<{\ensuremath{\langle}}
\def\>{\ensuremath{\rangle}}

\DeclareMathOperator{\UH}{UH}

 \DeclareMathOperator{\Hom}{Hom}

\DeclareMathOperator{\Span}{Span}
\DeclareMathOperator{\AffSpan}{AffSpan}

\DeclareMathOperator{\Face}{Face}
\DeclareMathOperator{\Conv}{Conv}
\DeclareMathOperator{\Supp}{Supp}
\DeclareMathOperator{\Star}{Star} 

\DeclareMathOperator{\Trop}{Trop}

\DeclareMathOperator{\trop}{trop}

\DeclareMathOperator{\Vol}{Vol}

\DeclareMathOperator{\st}{st}

\def\init#1#2{{{\operatorname{in}}_#1(#2)}}

\begin{document}

\title[Lifting Tropical Curves in Surfaces in $3$-space]{Obstructions to Lifting Tropical Curves in Surfaces in $3$-space}

\author[Bogart]{Tristram Bogart}
\address{Department of Mathematics, Universidad de los Andes, Bogot\'{a}, Colombia}
\email{tc.bogart22@uniandes.edu.co}

\author[Katz]{Eric Katz}
\address{Department of Combinatorics and Optimization, University of Waterloo, Waterloo, ON, Canada}
\email{eekatz@math.uwaterloo.ca}

\begin{abstract} Tropicalization is a procedure that takes subvarieties of an algebraic torus to balanced weighted rational complexes in space.  In this paper, we study the tropicalizations of 
 curves in surfaces in $3$-space.  These are balanced rational weighted graphs in tropical surfaces.   Specifically, we study the `lifting' problem: given a graph in a tropical surface, can one find a corresponding algebraic curve in a surface?   We develop specific combinatorial obstructions to lifting a graph by reducing the problem to the question of whether or not one can factor a polynomial with particular support in the characteristic $0$ case.  This explains why some unusual tropical curves constructed by Vigeland are not liftable.
\end{abstract}

\maketitle

\section{Introduction} \label{sec:intro}

Tropicalization is a procedure that takes an algebraic variety $V$
to a polyhedral complex $\Trop(V)$. Many algebraic properties of $V$ are reflected in the combinatorics of $\Trop(V)$.   In a certain sense, $\Trop(V)$ is special among polyhedral complexes in that it arose from an algebraic variety.  A natural question is {\em how special must a tropicalization be?}  In other words, {\em which  polyhedral complexes arise from algebraic varieties?}

Tropical geometry studies varieties over a field $\K$ with a non-trivial discrete
valuation $v:\K^*\rightarrow \Q$.  
If $V$ is a $d$-dimensional subvariety of the algebraic torus $(\K^*)^n$, then the tropicalization, $\Trop(V)\subset\R^n$ is a  
balanced weighted rational polyhedral complex of pure dimension $d$. 
Given such a complex
$\Gamma$, one may ask if there is an algebraic variety $V$
with $\Trop(V)=\Gamma$.  If so, we say $\Gamma$ \emph{lifts}.
This question is called the {\em lifting problem} in tropical geometry.  

The lifting problem is surprisingly subtle even for one-dimensional varieties.  A
one-dimensional balanced weighted rational polyhedral complex or 
\emph{tropical curve} is an edge-weighted graph in $\R^n$ with 
rational edge directions that satisfies a balancing condition at each vertex. Mikhalkin showed that balanced trees and more generally regular tropical curves always lift \cite{MikhICM}.  Speyer gave a sufficient condition for a
balanced graph with one cycle to lift to an algebraic curve of genus
one \cite{Spe07}.  His condition is also necessary when the
graph is trivalent and the residue field of $\K$ is of characteristic zero. The condition was
extended to certain higher-genus graphs by Nishinou \cite{N} and in 
forthcoming work of Brugall\'{e}-Mikhalkin and Tyomkin. The
second-named author has also found new obstructions for lifting tropical
curves in space \cite{KSpace}.   

In this paper, we consider the following two variants of the lifting problem:

\begin{definition} The {\em relative lifting problem} is the following: let $V$ be a variety in $(\K^*)^n$.  Let $\Gamma$ be a balanced weighted rational polyhedral complex contained in $\Trop(V)$.  Does there exist a subvariety $W\subset  V$ with $\Trop(W)=\Gamma$?
\end{definition}

\begin{definition} The {\em lifting problem for pairs} is the
  following: let $\Gamma$ and $\Sigma$ be balanced weighted rational
  polyhedral complexes with $\Gamma\subset\Sigma\subset\R^n$. Do there exist varieties $W\subset V\subset(\K^*)^n$ with $\Trop(W)=\Gamma$ and $\Trop(V)=\Sigma$?
\end{definition}

Specifically, we consider the case where $\Gamma$ is a tropical curve
and $V$ is a surface in $(\K^*)^3$ or $\Sigma$ is a tropical
surface in $\R^3$.  Answering lifting questions in this 
case is a prerequisite to using tropical curves to count classical
curves in surfaces in a way analogous to Mikhalkin's work on
curves in toric surfaces \cite{Mi03}.  We produce necessary conditions
for a pair consisting of a tropical curve $\Gamma$ in a unimodular tropical surface to
lift. In contrast to the previously mentioned results which gives
constraints on the structure of cycles in $\Gamma$, our conditions are
local: they depend only on the stars of the polyhedral complexes 
$\Gamma$ and $\Sigma$ at particular points. We use our conditions
to show that some very unusual tropical curves on tropical
surfaces exhibited by Vigeland \cite{Vigeland} do not lift.   

Our main result is the following: 

\begin{theorem} \label{maintheorem} Let
  $\Gamma\subset\Trop(V(f))$ be a tropical curve in a unimodular
  tropical surface in $\R^3$.  Suppose that $w$ is a vertex or an interior point of an edge of
  $\Trop(V(f))$ and that $\Star_w(\Gamma)$ spans a rational plane $U$.  If
  $\Gamma$ lifts in $V(f)$ then one of the following must hold: 
\begin{enumerate}
\item $\Gamma$ is locally equivalent to an integral multiple of the stable intersection $\Trop(V(f))\cap_{\st} U$, or
\item The stable intersection $\Trop(V(f))\cap_{\st} U$ contains a classical segment of weight $1$ with $w$ in its interior as a local tropical cycle summand at $w$. 
\end{enumerate}
\end{theorem}

The proof proceeds by a series of reduction steps.  By replacing a
variety $V$ by its initial degeneration $\init{w}{V}$, we reduce to
the constant-coefficient case where varieties are defined over a field
$\k$ with trivial valuation. In this case, a unimodular surface
become a plane. By intersecting $\init{w}{V}$ with an
appropriately chosen toric surface, we reduce the situation to a problem 
in one dimension smaller.  Specifically, our surface in space is replaced by a curve in a toric surface.  This curve is fewnomial: the Laurent polynomial defining it has very few terms.
This is very close in sprit to the work of Khovanski\u\i \cite{KhoNewton} on relating the geometry of hypersurfaces to their support.
The tropical graph that we want to lift will be contained in the tropicalization of that curve.  Any lift must be a component of that curve.  We then apply combinatorial obstructions to this happening.

There is an obvious combinatorial obstruction for one
curve in a $2$-dimensional algebraic torus to be a component of another. If $f,g\in\k[x_1^\pm,x_2^\pm]$,
then knowing the tropicalizations of $V(f)$ and $V(g)$ is equivalent to
knowing the Newton polygons $P(f)$ and $P(g)$. In particular, if $g$
is a factor of $f$, then $P(g)$ must be a Minkowski summand of
$P(f)$. This condition is equivalent to $\Trop(V(g))$ being a tropical cycle
summand of $\Trop(V(f))$.  However, there are deeper obstructions that come from looking at the support of $f$ and $g$, that is, the set of monomials in their
expressions.  Knowing this support, we are able to say, in certain cases, that $g$ cannot possibly divide $f$ even though $P(g)$ is a Minkowski summand of $P(f)$.  

\excise{Our main philosophy of this paper is that tropical geometry is in a certain sense a higher codimension generalization of the theory of Newton polytopes.  In the theory of Newton polytopes, one associates to a multivariable polynomial, 
\[f=\sum_{\U\in\ca} a_{\U} x^{\U}\] 
where $\ca\subset\Z^n$ is a finite set and $a_\U\in\k^*$, the polytope $P(f)=\Conv(\ca)$.  One may ask whether $f$ has a non-trivial factorization $f=gh$.  A natural condition is that $P(f)$ can be written as a Minkowski sum $P(g)+P(h)$.   But there are more subtle conditions placed on the support set $\ca$ rather than its convex hull.  For example, we must have $\Supp(f)\subset \Supp(g)+\Supp(h)$ and all points in $(\Supp(g)+\Supp(h))\setminus\Supp(f)$ must be expressible as a sum in $\Supp(g)+\Supp(h)$ in at least two different ways to allow coefficients to cancel.   In this paper, we will reduce the tropical lifting problem to the polynomial factorization problem and make use of more subtle factorization obstructions.  }

We were inspired by the approach to the absolute lifting problem for
curves presented in the work of Nishinou-Siebert \cite{NS} which itself is a log geometry analogue of the patchworking approach of Mikhalkin \cite{Mi03} following Viro \cite{Viro}.  The
heuristic for studying lifting problems is as follows.  First, one uses the tropicalization of a curve as a blueprint to
construct a degeneration of the ambient algebraic toric variety to a broken
toric variety. In each component of the broken toric variety, one
construct components of a broken curve.  Then one matchs the components together
to create a global broken object. Finally, one uses
deformation theory to extend the broken object to a smooth object in
the algebraic torus.  The obstruction we study is 
a failure to
construct the components of the broken curve. There are likely to be
further obstructions in the other steps.  In fact, all lifting obstructions for curves in space previously known to us occur in the deformation theory step.

Since the preprint of this  paper first appeared, Brugall\'{e} and Shaw released a preprint offering a different approach to lifting obstructions for tropical curves in smooth tropical surfaces \cite{BrSh}.  Their approach is based on intersecting a tropical curve with a canonical divisor and with a Hessian divisor.  Among other things, they are able to use their results to 
give a classification of liftable balanced weighted fans in tropical planes in $\R^3$.  

We now outline the remainder of the paper. Section 2 reviews tropical
background.  Section 3 studies the intersection and containment of tropical varieties and rational subspsaces.  Section 4 addresses combinatorial obstructions to a hypersurface being contained in another.  Section 5 uses these obstructions to study the lifting problems of curves in planes in the constant coefficient case.  Section 6 reduces the case of tropical curves in unimodular surfaces to
the case considered in section 5.  Section 7 applies the lifting
criteria to Vigeland's curves.  Section 8 provides combinatorial
results about bivariate polynomials that are needed in section
4. Section 9 which is logically independent from the rest of the
paper gives an example of a tropical curve that fails to lift in one
surface yet lifts in another surface with the same tropicalization. This
shows that the relative lifting problem does not have a purely 
combinatorial resolution and is not equivalent to the lifting problem
for pairs. 

We would like to thank Henry Segerman for producing the figures and
David Speyer for suggesting that we look at singularities of polynomials
with four monomials.  We would like to thank Erwan Brugall{\'e}, Ethan Cotterill, Mark Gross, Diane Maclagan, Gregg
Musiker, Sam Payne, Johannes Rau, Kristin Shaw, Bernd Siebert, and Bernd Sturmfels
for valuable conversations. The authors were partially supported by
NSF grant DMS-0441170 administered by the Mathematical Sciences
Research Institute (MSRI) while they were in residence at MSRI during
the Tropical Geometry program, Fall 2009.  We would like to thank the
program organizers and MSRI staff for making the program possible.

\section{Tropical Background} \label{sec:tropical}

In this section, we review relevant aspects of tropical geometry.  We
recommend \cite{BJSST,KTT,RST,Spe07} as further references. 

Let $\K$ be a field with non-trivial discrete valuation $v:\K^*\rightarrow G\subset\Q$.  Let $\ccr\subset\K$ be the valuation ring.  Let $\cm$ be the maximal
ideal of $\ccr$, and let $\k=\ccr/\cm$ be the residue field which we require to be algebraically closed and of characteristic $0$.   We write $t$
for the uniformizer and we
suppose for the sake of convenience that there is a splitting
$G\rightarrow \K^*$ which we denote by $a\mapsto t^a$.
For $x\in\ccr$, we write $x|_{t=0}$ for its image in the residue field.  In the sequel, one may take $\K=\C((t))$, the field of Laurent series with 
$v(c_at^a + \textup{higher order terms}) =a$.  In this case $\ccr=\C[[t]]$ and $\k=\C$.

Let $\ca \subset \Z^n$ be finite.  A Laurent polynomial
$f\in\K[x_1^\pm,\dots,x_n^\pm]$ with \emph{support} $\ca$ is one of the form  \[f=\sum_{\U\in\ca} a_\U \x^\U, a_\U\neq 0.\]
The \emph{Newton polytope} $P(f)$ of $f$ is the convex hull of $\ca$
in $\R^n$.   For $f\in\K[x_1^\pm,\dots,x_n^\pm]$ and $w\in G^n$, write 
 \[f(t^{w_1}x_1,\dots,t^{w_n}x_n)=t^h g(x_1,\dots,x_n)\]
 where $g\in \ccr[x_1^\pm,\dots,x_n^\pm]$ and no positive power of $t$ divides
 $g$.  Then the \emph{initial form} $\init{w}{f}$ is given by
 $\init{w}{f}=g|_{t=0}$. Given an ideal $I\subset \K[x_1^\pm,\dots,x_n^\pm]$, the \emph{initial ideal} $\init{w}{I}$ is the ideal given by
\[\init{w}{I}=\langle \init{w}{f}|f\in I \rangle.\]  For $X=V(I)$, a
variety in $(\K^*)^n$, the \emph{initial degeneration} $\init{w}{X}$
is the variety $V(\init{w}{I})\subset(\k^*)^n$.
 
The initial form of $f$ can be understood in terms of the {\em Newton
  subdivision} of $P(f)$. To obtain it, consider the upper hull 
\[\UH=\Conv(\{(\U,h)|\U\in\ca,h\geq v(a_\U)\}) \subset \R^n\times \R .\]
Projecting the faces of $\UH$ to $\R^n$ gives a subdivision of $P(f)$, all of whose vertices are points of $\ca$.
The support of $\init{w}{f}$ is a particular cell in this subdivision.  Specifically, let $F$ be the face of $\UH$ on which the function $l(\U,h)=\U\cdot w+ h$ is minimized.   Then the set $\ca_w=\{\U\in\ca|(\U,v(a_\U))\in F\}$ is the support of $\init{w}{f}$.

Let $\overline{\K}$ denote the algebraic closure of $\K$.  Then $v(\overline{\K})=\Q$.  Since every variety in $(\overline{\K}^*)^n$ is defined over some finite extension of $\K$, we may apply initial degenerations to varieties $X\subset(\overline{\K}^*)^n$.
We extend the valuation by Cartesian product to $v:(\overline{\K}^*)^n\rightarrow \Q^n\subset\R^n$.
The \emph{tropicalization}, $\Trop(X)$, of a pure $d$-dimensional
variety $X\subset(\K^*)^n$ is the topological closure of the set $v(X_{\overline{\K}})$ in $\R^n$.
The tropicalization is a pure $d$-dimensional rational polyhedral
complex in $\R^n$ \cite{BG}.  Each top-dimensional cell $F$ of
$\Trop(X)$ is assigned a natural positive integer multiplicity or \emph{weight} $m(F)$
under which the complex is \emph{balanced} (in some of the literature,
this is
referred to as the zero tension condition) \cite[Sec. 10]{Spe07}.  When we want to consider only the set $\Trop(X)$ and not its multiplicities, we refer to the {\em underlying set} of the tropicalization.
When $X$ is a curve, $\Trop(X)$ is a graph with rational edge
directions, and the balancing condition is as follows: if $v$ is a
vertex of $\Trop(X)$ with adjacent edges $e_1,\dots,e_k$ pointing in
primitive integer directions $u_1,\dots,u_k$, then \[\sum_{i=1}^k m(e_i)u_i=0.\]
In general, we will use the name {\em tropical variety} to refer to balanced positively-weighted rational polyhedral complexes whether or not they are tropicalizations.

The {\em constant coefficient case} is that of varieties defined over a field $\k$ with trivial valuation.
Given $X\subset(\k^*)^n$, we may define its tropicalization by setting
$\K'=\k((t))$ and defining $\Trop(X)=\Trop(X\times_{\k} \K')$.  In
this case, the tropicalization is a polyhedral fan in $\R^n$.

Tropicalization is functorial with respect to  monomial morphisms.  Let $h:(\K^*)^{n_1}\rightarrow (\K^*)^{n_2}$ be a homomorphism of algebraic tori.  Such a map is called a {\em monomial morphism}. Let 
\[h^\vee:\Hom(\K^*,(\K^*)^{n_1})\cong \Z^{n_1}\rightarrow\Hom(\K^*,(\K^*)^{n_2})\cong \Z^{n_2}\]
 be the induced map on one-parameter subgroup lattices. Then for
 $X\subset (\K^*)^{n_1}$, $\Trop(h(X))=h^\vee_\R(\Trop(X))$ where 
\[h^\vee_\R:\R^{n_1}\cong \Z^{n_1}\otimes \R\rightarrow\R^{n_2}\cong
\Z^{n_2}\otimes\R\] is induced by tensoring with $\R$. We call an invertible monomial morphism a {\em monomial change of variables}.  It induces an integral linear isomorphism on one-parameter subgroup lattices.
Any monomial morphism induces a dual map on character lattices
\[h^\wedge:\Hom((\K^*)^{n_2},\K^*)\cong \Z^{n_2}\rightarrow\Hom((\K^*)^{n_1},\K^*)\cong \Z^{n_1}\]

There is a natural  sum operation on tropical varieties.  Given two balanced positively-weighted polyhedral complexes of the same dimension, $\cd,\cd'$, we define their {\em tropical cycle sum} $\cd+\cd'$ to be the tropical variety such that: the underlying set is $\cd\cup\cd'$; the polyhedral structure restricts to a refinement of the given polyhedral structures on $\cd,\cd'$; the weight of a top-dimensional cell in $\cd+\cd'$ is the sum of the weights on the cells containing it in $\cd$ and $\cd'$ (where the weight on a cell that does not belong to a given complex is taken to be $0$).  Note that tropical cycle sum is defined only up to refinement.

Tropicalization behaves well with respect to initial degenerations.
The fundamental theorem of tropical geometry \cite{PFibers} states
that for a point $w \in \Q^n$, $w\in\Trop(X)$ if and only if
$\init{w}{X}\neq\emptyset$. For $w\in\Trop(X)$, the \emph{star} of $X$
at $w$, denoted $\Star_w(\Trop(X))$, is the set of all $v\in\R^n$ such
that $w+\epsilon v\in\Trop(X)$ for all sufficiently small
$\epsilon>0$.  Then, $\Trop(\init{w}{X})=\Star_w(\Trop(X))$.  Note that $\Star_w(\Trop(X))$ inherits the structure of a balanced weighted fan from the balanced weighted polyhedral structure on $\Trop(X)$.   Looking at stars allow us to consider polyhedral complexes locally.  We will say that two weighted complexes $\cd,\cd'$ are
{\em locally equivalent} at a point $w$ if, after possible refinement, $\Star_w(\cd)$ and
$\Star_w(\cd')$ are identical.  Note that tropical cycle sum commutes with taking the star at a point $w$.  We say a weighted complex $\mathcal{E}$ is the local tropical cycle sum of $\cd,\cd'$ at $w$ if $\Star_w(\mathcal{E})=\Star_w(\cd)+\Star_w(\cd')$.

 
Tropicalizing hypersurfaces in $(\K^*)^n$ is straightforward.   For
$f$ a Laurent polynomial with support set $\ca$, we define the {\em tropical polynomial}
\[\trop(f)(w)=\min_{\U\in\ca} \left(w\cdot\U+v(a_\U)\right)\]
 which is a piecewise linear function. As a consequence of Kapranov's
 theorem  \cite[Thm 2.1.1]{EKL}, $\Trop(V(f))$ is equal to the corner
 locus of $\trop(f)$: that is, the subset of $\R^n$ on which the
 minimum is achieved by at least two values of $\U$. In the constant-coefficient case, 
$\Trop(V(f))$ is the positive codimension skeleton of the inner normal fan to the Newton polytope
$P(f)$. In particular, a top-dimensional cone in $\Trop(V(f))$
corresponds to an edge of $P(f)$, and its
multiplicity is the lattice length of that edge.

The theory of balanced codimension $1$ fans is a dual reformulation of the theory of Newton polytopes.   The weighted normal fan of the Minkowski sum of rational polytopes $P+Q$ is the tropical cycle sum of the normal fans $\cn(P),\cn(Q)$.  Therefore, if $P(g)$ is a Minkowski summand of $P(f)$ then $\Trop(V(g))$ is contained in $\Trop(V(f))$ as a tropical cycle summand. 

When we restrict to Newton polytopes in the plane, the situation becomes even simpler.  Any balanced positively-weighted rational codimension $1$ fan is the weighted normal fan to an integral polygon.  The edges of the polygon are normal to the rays of the fan; their order is given by the order of the rays of the fan traversed counterclockwise about the origin; their lattice lengths are given by the weights; and the condition that the edges close up is the balancing condition.  The tropicalization of a curve with that polygon as its Newton polytope will be the original fan.
Thus we get our first lifting result: any codimension $1$ balanced positively-weighted rational fan in $\R^2$ is the tropicalization of a curve in $(\k^*)^2$.
 

A {\em subtorus} $i:T\hookrightarrow (\K^*)^n$ is a monomial inclusion of $T\cong (\K^*)^{n'}$ such that $i^\vee(\Z^{n'})$ is a saturated sublattice of $\Z^n$.   We will also use subtorus to refer to inclusions of the form $i:T\hookrightarrow (\k^*)^n$ with $T\cong (\k^*)^{n'}$ satisfying the above property.  We will identify $T$ with its image under $i$ in the sequel.
   If $T$ is a subtorus of $(\K^*)^n$, then $\Trop(T)$ is the linear subspace 
\[\Trop(T)=\Hom(\K^*,T)\otimes\R\subset \Hom(\K^*,(\K^*)^n)\otimes\R\cong\R^n.\]
If $z\in (\K^*)^n$ then
\[\Trop(z\cdot X)=\Trop(X)+v(z)\]
where $\cdot$ denotes multiplication in the torus.

Tropicalization is linear on the underlying cycles of subschemes of
$(\K^*)^n$.  In other words, if $V$ is an irreducible non-reduced
scheme of length $k$ with reduction $\operatorname{Red}(V)$, then $\Trop(V)$ has the same underlying set as
$\Trop(\operatorname{Red}(V))$ with weights multiplied by $k$.
Additionally, if $V$ is a $d$-dimensional subscheme of $(\K^*)^n$ with
irreducible components $V_1,\dots,V_l$, then $\Trop(V)$ is the tropical cycle sum of the $\Trop(V_i)$'s.
  
Tropicalization does not always commute with intersection.
 It is true in general that $\Trop(X\cap Y)\subseteq \Trop(X)\cap \Trop(Y)$. There is a useful necessary condition for the reverse inclusion. 
Two tropical varieties $\Trop(X),\Trop(Y)$ are said to {\em intersect
  transversely at a point $w\in\Trop(X)\cap\Trop(Y)$} if $w$ is in the
relative interior of cells $F,G$ of $\Trop(X),\Trop(Y)$, respectively
such that $\Span_\R(F-w)+\Span_\R(G-w)=\R^n$. By extension, $\Trop(X)$
and $\Trop(Y)$ are said to {\em intersect transversely} if every point
of intersection of $\Trop(X)$ and $\Trop(Y)$ is a transverse
intersection point.  The transverse intersection lemma \cite[Lemma
3.2]{BJSST} states that if $w\in\Trop(X)\cap\Trop(Y)$ is a 
transverse intersection point, then $w\in\Trop(X\cap Y)$.  It follows that if $\Trop(X)$ and $\Trop(Y)$ intersect transversely, then $\Trop(X\cap Y)=\Trop(X)\cap\Trop(Y)$.
If $w$ is a point in a top-dimensional cell of $\Trop(X\cap Y)$ lying in transverse cells $F$, $G$ of $\Trop(X)$, $\Trop(Y)$, respectively, then the multiplicity of $\Trop(X\cap Y)$ at $w$ is the product $m_X(F)m_Y(G)[\Z^n:\Span_\Z(F-w)+\Span_\Z(G-w)]$.

In the case where the intersection is not transverse, we can still
define the \emph{stable intersection} \cite{RST,MikhICM}. Let $v$ be
a generically chosen vector in $\R^n$.  Then $\Trop(X)$ and
$\Trop(Y)+sv$ intersect transversely for small $s>0$.  The stable
intersection is defined to be the Hausdorff limit 
\[\Trop(X)\cap_{\st}\Trop(Y)=\lim_{s\to 0}\Trop(X)\cap
\left(\Trop(Y)+sv\right).\]  
This definition turns out to be independent of the choice of $v$.  By
adding weights when top-dimensional cells coincide in the limit,
we obtain weights on $\Trop(X)\cap_{\st}\Trop(Y)$.  Note that stable intersection commutes with taking the star at a point $w$.

The stable intersection of a subtorus and hypersurface is easy to compute.

\begin{lemma} \label{stablepoly} Let $f\in\k[x_1^\pm,\dots,x_m^\pm]$ be a polynomial with support set $\ca$ and let $i:T\hookrightarrow(\k^*)^n$ be a subtorus.  The stable intersection of $\Trop(V(f))$ and $\Trop(T)$ (viewed as a subcomplex of $\Trop(T)$) is the positive codimension skeleton of the normal fan to $\Conv(i^\wedge(\ca))$. 
\end{lemma}

\begin{proof}
By \cite[Prop 2.7.4]{OP}, for general $z\in (\k^*)^n$,
\[\Trop(V(f))\cap_{\st}\Trop(z\cdot T)=\Trop(V(f)\cap z\cdot T).\]
Now, $V(f)\cap z\cdot T$ is cut out from $z\cdot T$ by a polynomial whose support set is $i^\wedge(\ca)$.    Consequently, $\Trop(V(f)\cap z\cdot T)$ is the positive codimension skeleton of $\cn(\Conv(i^\wedge(\ca)))$.
\end{proof}

\section{Tropical Varieties and Classical Subspaces}

In this section, we find conditions for tropical varieties and rational linear subspaces to be contained in each other.
 
Given a rational linear subspace $U\subset \R^n$, we may pick a 
subtorus $i:T\hookrightarrow(\k^*)^n$ with $\Trop(i(T))=U$.  A \emph{subtorus-translate} $H$ is a variety of the form $z\cdot T$ where 
$T$ is a subtorus and $z\in(\k^*)^n$.  Note that $\Trop(T)=\Trop(H)$. We will use $H^\wedge=\Hom(T,\k^*)$ to denote the 
character lattice of $T$.  The inclusion $i$ induces a projection $i^\wedge:\Z^n\rightarrow H^\wedge$ of 
character lattices.  

The following lemma will be used to reduce the dimension of certain lifting problems.

\begin{lemma} \label{subtori} Let $W\subset (\k^*)^n$ be an irreducible and reduced subvariety.  Let $i:T\hookrightarrow (\k^*)^n$ be a  subtorus.   If $\Trop(W)\subset \Trop(T)$ then there exists $z\in(\k^*)^n$ such that  $W\subset z\cdot T$. 
\end{lemma}

\begin{proof}
Let $\U\in \Hom((\k^*)^n,\k^*)$ be a character constant on $T$.  Then $\U$ is a monomial morphism and induces a map
$\U^\vee_\R:\R^n\rightarrow\R$.  Since
$\Trop(\U(W))=\U_\R^\vee(\Trop(W))$ is $0$-dimensional, $\U(W)$
is also $0$-dimensional, hence a point.  Consequently, $\U$ is
equal to a constant $z_\U\in\k$ on $W$. That is, $W$ is contained in the subtorus translate defined by $\U=z_\U$.
By applying this argument to the characters
cutting out $T$, we find that $W$ is contained in a translate of $T$.
\end{proof}

Let $f\in \k[x_1^\pm,\dots,x_n^\pm]$ be a Laurent polynomial given by
\[f=\sum_{\U\in\ca} a_\U x^\U\]
where $\ca\subset \Z^n$ is a finite support set and each $a_\U\neq 0$.  
Then $V(f)$ is a hypersurface in $(\k^*)^n$.  Its tropicalization
$\Trop(V(f))$ is the positive codimension skeleton of the inner normal fan to the Newton polytope $\Conv(\ca)$.  

Let $i:T\rightarrow(\k^*)^n$ be a subtorus.    We consider how $\Trop(V(f))$ can intersect the tropicalization of a torus, $\Trop(T)$.  We say that two polyhedral complexes intersect properly if $\dim(\Trop(V(f))\cap \Trop(T)))=\dim(\Trop(V(f)))+\dim(\Trop(T))-n$.
We consider the following possibilities:
\begin{enumerate}
\item $\Trop(T)$ is contained in $\Trop(V(f))$,
\item $\Trop(T)$ intersects $\Trop(V(f))$ non-properly, and
\item $\Trop(T)$ intersects $\Trop(V(f))$ properly.
\end{enumerate}
These geometric situations have algebraic consequences in terms of the support set $\ca$.  We consider the case of $\Trop(T)$ intersecting a certain tropical hypersurface properly.

\begin{lemma} \label{lemma:properint} Let $\ca$ be a set of $n+1$ affinely independent points in $\Z^n$.  Let $f$ be a Laurent polynomial with support set $\ca$.  Then the following are equivalent:
\begin{enumerate}
\item \label{i:inj} $i^\wedge$ is injective on $\ca$,
\item \label{i:prop} $\Trop(V(f))$ intersects $\Trop(T)$ properly,
\item  \label{i:trans} the top-dimensional cells of $\Trop(V(f))$ and $\Trop(T)$ meet transversely.
\end{enumerate}
In this case, $\Supp(i^*f)=i^\wedge(\ca)$ and $\Trop(V(f)\cap T)$ is the codimension $1$ skeleton of the normal fan to $\Conv(i^\wedge(\ca))$ in $\Trop(T)$.
\end{lemma}

\begin{proof}
Note that $\dim(\Trop(V(f)))+\dim(\Trop(T))-n=\dim(T)-1$. Therefore, the intersection of $\Trop(V(f))$ and $\Trop(T)$ is non-proper if and only if a top-dimensional cell of the intersection is of dimension $\dim(T)$.  It follows that (\ref{i:prop}) and (\ref{i:trans}) are equivalent.

Suppose $i^\wedge(\U_1)=i^\wedge(\U_2)$ for some $\U_1,\U_2\in\ca$.  Then $\Trop(T)$ is orthogonal to $\U_1-\U_2$.  This, in turn, is equivalent to $\Trop(T)$ containing the face $F$ dual to the edge $\{\U_1,\U_2\}$ of $\Conv(\ca)$.  This implies that the intersection of $\Trop(T)$ and a top-dimensional face of $\Trop(V(f))$ is non-transverse, showing that (\ref{i:trans}) implies (\ref{i:inj}).  The converse is similar.

Suppose $i^\wedge$ is injective on $\ca$.  Then each exponent in $i^*f$ corresponds to a unique exponent of $f$. Consequently, the support of $i^*f$ is exactly $i^\wedge(\ca)$.  In that case, the tropicalization of $V(f)\cap T$ is $\Trop(V(i^*f))$ which is the positive codimension skeleton of the normal fan to $\Conv(i^\wedge(\ca))$. 
\end{proof}

Now, we turn to the case when a torus is contained in a hypersurface $V(f)$.  We get lifting obstructions by considering the support of $f$.

\begin{lemma} \label{containmentlemma} Let $T$ be a subtorus contained in $V(f)$.  Let $i^\wedge:(\Z^n)^\wedge\rightarrow T^\wedge$ be the natural projection.
Then $i^\wedge: \ca\rightarrow T^\wedge$ has no fibers consisting of a single point.  
\end{lemma}

\begin{proof}
Let $i:T\hookrightarrow (\k^*)^n$ be the inclusion.  Because $T$ is
contained in $V(f)$, $i^*(f)=0$.  The support of the polynomial $i^*(f)$ is contained in $i^\wedge(\ca)$.  If $\V\in i^\wedge(\ca)$, the coefficient of $\V$ in $i^*f$ is zero.  However, this coefficient is a linear combination with non-zero coefficients of the $a_\U$'s for $\U\in (i^\wedge)^{-1}(\V)$.  For this linear combination to be equal to $0$, 
 $(i^\wedge)^{-1}(\V)$ cannot consist
of a single element.
\end{proof}

The above lemmas can be used to give non-existence results for lifting rational subspaces.  If $V(f)$ is a hypersurface and $U$ is a rational linear subspace contained in $\Trop(V(f))$, the reduction of any irreducible lift of $U$ must be a subtorus-translate by Lemma \ref{subtori}.  But such a subtorus-translate may be excluded by Lemma \ref{containmentlemma}.

\begin{example} Let $f(x_1,x_2,x_3)=1+x_1^2+x_2^2+x_3^2+x_1x_2$.  Observe that $P(f)$
is a dilate of the Newton polytope of $g=1+x_1+x_2+x_3$.  Consequently,
the underlying set of $\Trop(V(f))$ is the standard tropical plane.
Let $\Gamma$ be the line in $\Trop(V(f))$ passing through $(0,0,0)$
and $(0,1,1)$.  Because it is the tropicalization of the classical
line $\{x_1=-1,x_2=-x_3\}$, it lies on $\Trop(V(g))$ and hence on $\Trop(V(f))$.   Any lift of $\Gamma$ must be supported on a one-dimensional subtorus-translate $i:T\hookrightarrow (\k^*)^3$.  Therefore $i(T)\subset V(f)$.
Pick a coordinate on $\Z=\Hom(T,\k^*)$ such that $i^\wedge:\Z^3\rightarrow \Z$ is given by 
\[e_1\mapsto 0,\ e_2\mapsto 1,\ e_3\mapsto 1\]
where $e_1,e_2,e_3$ are the standard basis vectors for $\Z^3$.  Then
$i^\wedge(\ca)=\{0,1,2\}$, but $(i^\wedge)^{-1}(1)=\{(1,1,0)\}$, a
singleton.  Therefore $\Gamma$ does not lift in $V(f)$
\excise{Equivalently, one can also understand this case by parameterizing a potential lift of $\Gamma$ by 
\[t\mapsto (a,bt,ct)\]
for $a,b,c\in\k^*$.  Then $i^*(f)=1+a^2+b^2t^2+c^2t^2+abt$.  Since $t$ has a non-zero coefficeint, this polynomial cannot vanish everywhere. }

Note that $\Gamma$ does lift in $V(g^2)$ even though $\Trop(V(f))=\Trop(V(g^2))$.  This example shows that the relative lifting problem is not combinatorial in the sense that one cannot determine if $\Gamma$ lifts from knowing only $\Trop(V(f))$.  This illustrates the difference between the relative lifting problem and the lifting problem for pairs.
\end{example}

\section{Components of Hypersurfaces}

Now, we consider the case where $f$ is a Laurent polynomial over $\k$ with support $\ca\subset\Z^n$.  Let $\Sigma$ be a tropical hypersurface in $\R^n$ contained in $\Trop(V(f))$.  We find necessary conditions for $\Sigma$ to be the tropicalization of an irreducible component of $V(f)$.   Write $\Sigma=\Trop(V(g))$ for some Laurent polynomial $g$ where the Newton polytope of $g$ is determined by $\Sigma$.   Now, $V(g)$ is a component of $V(f)$  if and only if $g$ divides $f$.  A necessary condition for $g$ to divide $f$ is for $\Trop(V(g))$ to be a tropical cycle summand of $\Trop(V(f))$  This condition should be thought of as the coarsest combinatorial condition. We will use a finer combinatorial condition based on a strategy suggested to us by David Speyer.  If $f$ is reducible then either $f$ is generically non-reduced or $V(f)$ has several connected components.  The non-reducedness can often be ruled out by examining $\ca$ and similarly, if $V(f)$ has several irreducible components, they can often be shown to intersect in singular points of $V(f)$.  The number and types of singularities of $V(f)$ are constrained by the support of $f$.

We first consider the special case where $\ca\subset \Z^n$ has convex hull a simplex.  In this case, $\Trop(V(f))$ is combinatorially equivalent to a tropical hyperplane but with possibly non-trivial multiplicities on the faces.  The only tropical hypersurfaces contained in $\Trop(V(f))$ as tropical cycle summands are copies of $\Trop(V(f))$ with possibly smaller multiplicities on each face.  We show that under 
certain conditions on $\ca$, the only possibility for such a hypersurface is $\Trop(V(f))$.

\begin{lemma} \label{l:4triangle} Let $\ca\subset\Z^n$ be a  set satisfying
\begin{enumerate}
\item $|\ca|\leq 2n$,
\item the convex hull of $\ca$ is an $n$-dimensional simplex, and  
\item $\AffSpan_\Z(\ca)=\Z^n$. 
\end{enumerate}
If $f$ is a Laurent polynomial with support set $\ca$, then $f$ is irreducible.
\end{lemma}

\begin{proof}
Because the only non-trivial Minkowski summands of $P(f)=\Conv(\ca)$ are dilates, for $f$ to have a non-trivial factorization, $P(f)$ must be a $k$-fold dilate of some simplex for $k\geq 2$.   By multiplying by a monomial, we may ensure that $P(f)$ has a vertex at the origin.  From this, we have $\AffSpan_\Z(\ca)=\Span_\Z(\ca)$.  Consequently, the integer span of the other vertices of $P(f)$ is a lattice $\Lambda$ in $\Z^n$ contained in $k\Z^n$.  The $|\ca|-(n+1)$ non-vertex points of $\ca$ cannot span $\Z^n/\Lambda$.  It follows that the affine span of $\ca$ is not $\Z^n$.  This contradiction shows that there can be no non-trivial factorization.
\end{proof}

\noindent Now we restrict ourselves to the case where $n=2$.  We apply the following lemma for which we make no claims to originality.

\begin{lemma} \label{l:3triangle} Suppose $\ca\subset\Z^2$ consists of three non-collinear points.  Then $f$ is irreducible.
\end{lemma}

\begin{proof}
We claim that $V(f)$ is a smooth curve.  By a monomial change of variables applied to $(\k^*)^2$, we may ensure that two points $\U_1,\U_2$ of $\ca$ lie on a line parallel to the $x_2$-axis.  Furthermore, we may translate $\ca$ by dividing $f$ by a monomial, and ensure that $\U_1,\U_2$ both have $x_1$ exponent equal to $0$.  Consequently, $\frac{\partial f}{\partial x_1}$ is a monomial and is never $0$ in $(\k^*)^n$.  It follows that $V(f)$ is smooth.

The only Minkowski summand of the triangle $P(f)$ is a dilate of $P(f)$.  Consequently, if $f=gh$ is a non-trivial factorization, then $P(g)$ and $P(h)$ are both dilates of some simplex $P$.  Write $P(g)=n_gP$, $P(h)=n_hP$. Lemma~\ref{lem:intersectionintorus} applies for all $w\in\Q^2\setminus\{0\}$.  Therefore, by Bernstein's theorem as stated as Theorem \ref{t:bernstein}, $V(g)$ and $V(h)$ must intersect in $(\k^*)^2$ with multiplicty $n_gn_h\cdot\Vol(P)$.  This implies $V(f)$ is singular, giving a contradiction.
\end{proof}

Now, we consider a case where $|\ca|=4$. The following lemma follows from Proposition \ref{prop:binomialfactor}.

\begin{lemma}  \label{l:containsaline} Suppose $\ca\subset\Z^2$ be a set with $|\ca|=4$ whose convex hull is a quadrilateral and $\AffSpan_\Z(\ca)=\Z^2$.  If $f$ is reducible then $\Trop(V(f))$ contains a classical line of weight $1$ as a tropical cycle summand.
\end{lemma}

\begin{example} \label{ex:fourterm}
Let 
\[f(x,y)=A+Bx+Cy^3+Dx^2y\]
be a polynomial with $A,B,C,D\in\k^*$.
Then $\Trop(V(f))$ contains a standard tropical line $\Gamma=\Trop(V(1+x+y))$ as depicted in Figure \ref{f:fourterm}.  However, since $\Trop(V(f))$ does not contain a classical line, $\Gamma$ does not lift to a component of $\Trop(V(f))$ for any choice of $A,B,C,D$.  The restriction on $\ca$ is necessary here: the polynomial 
\[g(x,y)=(1+x+y)(1+y^2)=1+x+y+y^2+xy^2+y^3\]
 has the same Newton polygon as $f$, and $\Gamma$ lifts in $\Trop(V(g))$ by construction.
 
\begin{figure}
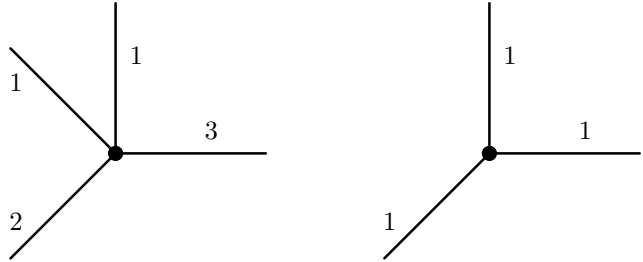
 \begin{center}
\begin{texdraw}
       \drawdim cm  \relunitscale 1
       \linewd 0.03
       \move(0 0) \fcir f:0 r:0.1
       \lvec(2 0)
       \move(0 0)
       \lvec(0 2)
       \move(0 0)
       \lvec(-1.4 1.4)
       \move(0 0)
       \lvec(-1.4 -1.4)
       \htext(1.2 .2){\footnotesize $3$}
       \htext(.2 1.2){\footnotesize $1$}
       \htext(-1.4 .85){\footnotesize $1$}
       \htext(-1.4 -1){\footnotesize $2$}
        
      \end{texdraw}
      \hspace{.5in}
\begin{texdraw}
       \drawdim cm  \relunitscale 1
       \linewd 0.03
       \move(0 0) \fcir f:0 r:0.1
       \lvec(2 0)
       \move(0 0)
       \lvec(0 2)
       \move(0 0)
       \lvec(-1.4 -1.4)
       \htext(1.2 .2){\footnotesize $1$}
       \htext(.2 1.2){\footnotesize $1$}
       \htext(-1.4 -1){\footnotesize $1$}
      \end{texdraw}
\end{center}
\caption{The tropical variety of $V(f)$ and $\Gamma$} \label{f:fourterm}
\end{figure}

\end{example}

\section{Lifting obstructions on tropical planes in $3$-space} \label{sec:obstructions}

In this section, we study the obstructions for lifting tropical curves on tropical planes in $\R^3$.  Let $\ca\subset\Z^3$ be a subset of the vertices of a unimodular simplex in $\R^3$.  Let $f\in\k[x_1^\pm,x_2^\pm,x_3^\pm]$ be a Laurent polynomial with support  $\ca$.  Then the tropical surface $\Trop(V(f))\subset \R^3$ is well understood.  If $|\ca|=1$, then $\Trop(V(f))$ is empty.    If $|\ca|=2$, then $\Trop(V(f))$ is a rational plane.   If $|\ca|=3$ then $\Trop(V(f))$ is (up to integral linear
isomorphism) of the form $\Trop(L)\times\R$ where $\Trop(L)$ is the tropicalization of a generic line in the plane.  It consists of three half-planes meeting along a line.  If $|\ca|=4$, then $\Trop(V(f))$ is a generic tropical plane in space.  By a monomial change of variables and the action by an element of $(\k^*)^3$, $f$ can be put in the form  $f=1+x+y+z$.  Then $\Trop(V(f))$ is the positive codimension skeleton of the normal fan to the standard unimodular simplex.  It has rays in the directions $e_1,e_2,e_3,-e_1-e_2-e_3$ and a cone spanned by each pair of those rays.  There are three lines in $\Trop(V(f))$ which all pass through the origin and are given by the direction vectors $e_1+e_2$, $e_1+e_3$, and $e_2+e_3$.  

Let $\Gamma\subset\Trop(V(f))\subset\R^3$ be a one-dimensional balanced positively-weighted rational fan.  We will study conditions for lifting $\Gamma$ to a curve in $V(f)$ defined over $\k$.   If $|\ca|=2$, then  $\Gamma$ is a positively-weighted codimension $1$ fan in a rational plane.  Therefore, it always lifts to a curve.
In the cases $|\ca|=3,4$, and the linear span of $\Gamma$ is not $\R^3$, we will apply Lemma \ref{subtori} to reduce to the cases studied in the previous section.

 \begin{definition} The tropical curve $\Gamma$ is said to be planar (resp. linear) at the origin if the linear span of $\Gamma$ is a plane (resp. line).
\end{definition}

Note that any trivalent one-dimensional balanced fan is planar.  We may rephrase  Lemma \ref{subtori} as the following:

\begin{lemma} \label{l:picksubtorus} Let $\Gamma\subset\Trop(V(f))$ be a tropical curve.  If $\Gamma$ lifts to an irreducible and reduced curve $C\subset V(f)$, then there exists a subtorus translate $z\cdot T$ for $z\in (\k^*)^n$ such that 
$\Trop(T)=\Span_\R(\Gamma)$ and $C\subset V(f)\cap z\cdot T$.  
\end{lemma}

In the case where $\Gamma$ is linear and lifts to a curve $C$ then the reduction of $C$ must be a one-dimensional torus translate.  In particular, $\Gamma$ lifts to an irreducible and reduced curve if and only if $\Gamma$ is of weight $1$.


We now consider the case where $\Gamma$ is planar.  The main result of this section is the following:

 \begin{proposition} \label{prop:starsin3space} Let $f\in\k[x_1^\pm,x_2^\pm,x_3^\pm]$ be a Laurent polynomial with support set $\ca$ consisting of at least three points of a unimodular simplex in $\R^3$.   Let $\Gamma$ be a planar balanced weighted $1$-dimensional fan in $\Trop(V(f))$.  Suppose $\Gamma$ lifts to an irreducible and reduced curve $C$ in $V(f)$.  Then, one of the following must hold (where $\Span_\R(\Gamma)$ is considered as a balanced fan with multiplicity $1$):
\begin{enumerate}
\item $\Gamma$ is equal to $\Trop(V(f))\cap_{\st} \Span_\R(\Gamma)$, or \label{isstabint}
\item $\Trop(V(f))\cap_{\st} \Span_\R(\Gamma)$ contains a classical line of weight $1$ as a  tropical cycle summand. \label{classicalsegment}
\end{enumerate}
\end{proposition}

Note that the condition $\Gamma=\Trop(V(f))\cap_{\st}\Span_\R(\Gamma)$ puts conditions on the multiplicities on edges of $\Gamma$ even when the condition that $\Gamma\subset \Trop(V(f))\cap \Span_\R(\Gamma)$ determines $\Gamma$ set-theoretically.  Using the fact that a non-degenerate tropical plane in $\R^3$ contains only three classical lines, Brugall\'{e} and Shaw have classified all liftable graphs in case (\ref{classicalsegment}) \cite[Thm 6.1]{BrSh}.  

We now prove the proposition.  Suppose that $\Gamma$ lifts to a curve $C$.  Lemma \ref{l:picksubtorus} guarantees a two-dimensional subtorus-translate $z\cdot T$ containing $C$.   Let $i_z:T\rightarrow (\k^*)^n$ be the inclusion of $T$ followed by translation by $z$.  Then $i_z^{-1}(C)$ is a subvariety of $V(i_z^*f)$.  If $V(i_z^*f)$ is a curve in the toric surface $T$, we can use the results of the above section to provide lifting obstructions.  The first step is to determine the support of $i_z^*f$.  A priori, we know only that the support is contained in $i_z^\wedge(\ca)$ because several monomials in $f$ may map by $i_z^\wedge$  to the same monomial leading to a cancellation among the coefficients of $i_z^*f$.  Fortunately, this does not occur by the following lemma:

\begin{lemma} If $\Trop(V(f))$ is planar, then $\Supp(i_z^*f)=i_z^\wedge(\ca)$ and $P(i_z^*f)$ is a polygon.  Furthermore, $\Trop(V(i_z^*f))=(i^\vee)^{-1}(\Trop(V)\cap_{\st} \Trop(T))$.
\end{lemma}

\begin{proof}
Suppose that  $\Supp(i_z^*f)\neq i_z^\wedge(\ca)$.  Then we must have that $i_z^\wedge$ is not injective on $\ca$.  It follows that $|\Supp(i_z^*f)|\leq |i_z^\wedge(\ca)|-1\leq |\ca|-2$.   We have $|\Supp(i_z^*f)|>0$ because $T\not\subseteq V(f)$.

If $|\Supp(i_z^*f)|=1$ then $V(i_z^*f)$ is empty.  

If $|\Supp(i_z^*f)|=2$, then $V(i_z^*f)\subset T$ is a toric curve.  Consequently $\Trop(V(i_z^*f))$ and hence $\Gamma$ are classical lines which contradicts $\Gamma$ being planar.

  If $P(i_z^*f)=\Conv(\Supp(i_z^*f))$ is not a polygon, then it must be a segment.  In that case, $\Trop(V(i_z^*f))$ would be a classical line which is again excluded.  

The final statement follows from Lemma \ref{stablepoly}.  
\end{proof}

It follows that $V(i_z^*f)$ is one-dimensional and $C$ is a component of it.
Now, either $C=V(i_z^*f)$ or $i^*(f)$ is reducible.   If $|\ca|=3$, then by Lemma \ref{l:3triangle}, $i_z^*f$ must be irreducible.   

Now, suppose $|\ca|=4$.  Since $i$ induces an inclusion on one-parameter subgroup lattices with image a saturated sublattice, by standard homological algebra, $i^\wedge$ is surjective.  Therefore, the affine span of $i^\wedge(\ca)$ is $\Z^2$.  
If $P(i_z^*f)$ is a triangle, then $i_z^*f$ is irreducible by Lemma \ref{l:4triangle}.  If $P(i_z^*f)$ is a quadrilateral then the theorem follows from Lemma \ref{l:containsaline}.

\section{Lifting Obstructions on Unimodular Surfaces} \label{sec:reduction}
In this section, we study lifting obstructions for tropical curves in
surfaces in $(\K^*)^3$.  Our strategy is to reduce to the constant coefficient case by using initial degenerations.
In fact, our necessary condition amounts to the following: if $\Gamma\subset\Trop(V(f))$ lifts then  for all $w\in\Trop(V(f))$, the fan $\Star_w(\Gamma)\subset \Trop(V(\init{w}{f}))$ lifts to a curve in $V(\init{w}{f})$.  The only caveat in passing to initial degenerations is that $\init{w}{C}$ may not be irreducible and reduced.  Therefore, we must apply the results from the previous section to the reduction of each component of $\init{w}{C}$.  Our results are summarized in Proposition \ref{prop:onKvertex} which is a rephrasing of Theorem \ref{maintheorem}.
   
\begin{definition} A Laurent polynomial $f\in\K[x_1^\pm,\dots,x_n^\pm]$ is said to be {\em unimodular} if the Newton subdivision of $P(f)$ is unimodular.
\end{definition}

Unimodularity is equivalent to the statement that for any
$w\in\Trop(V(f))$, $\init{w}{V(f)}$ is the image of a hyperplane under a monomial change of variables.  Then we may use the results from the previous section.
Note that if $w$ is in a $2$-cell, edge, or vertex of $\Trop(V(f))$,
then $\init{w}{f}$ consists of $2$, $3$, or $4$ monomials, respectively.  In each
case, $P(\init{w}{f})$ is a unimodular simplex.
We say a tropical curve $\Gamma$ is planar at a point $w$ if the linear span of $\Star_w(\Gamma)$ is a plane.

\begin{proposition} \label{prop:onKvertex} Let $\Gamma\subset\Trop(V(f))$ be a tropical curve in a unimodular hypersurface.  Suppose that $w$ is a vertex or an interior point of an edge of $\Trop(V(f))$ and $\Star_w(\Gamma)$ spans a rational plane $U$.  If $\Gamma$ lifts in $V(f)$ and $\Trop(V(f))\cap_{\st} U$ does not contain a classical segment of weight $1$ with $w$ in its interior as a local tropical cycle summand at $w$, then $\Gamma$ is locally equivalent at $w$ to an integral multiple of $\Trop(V(f))\cap_{\st} U$. 
\end{proposition}

\begin{proof}
If $\Gamma$ lifts to a curve $C$ in $V(f)$, then $\Star_w(\Gamma)$ lifts to $\init{w}{C}$ in $V(\init{w}{f})$.  Each irreducible component $C_i$ of $\init{w}{C}$ tropicalizes to a tropical cycle summand $\Gamma_i$ of $\Star_w(\Gamma)$.  Since $\Trop(V(\init{w}{f}))\cap_{\st} U$ does not contain any classical lines, the tropicalization of the reduction of each $C_i$ must be equal to $\Trop(V(\init{w}{f}))\cap_{\st} U$ by Proposition \ref{prop:starsin3space}.  
Consequently, each $\Gamma_i$ must be equal to an integral multiple of $\Trop(V(\init{w}{f}))\cap_{\st} U$.
\end{proof}

\section{Tropical Lines on Vigeland's Surface} \label{sec:vigeland}
 
\begin{figure}\begin{center}
\vspace{-.2in}
\includegraphics[height=90mm]{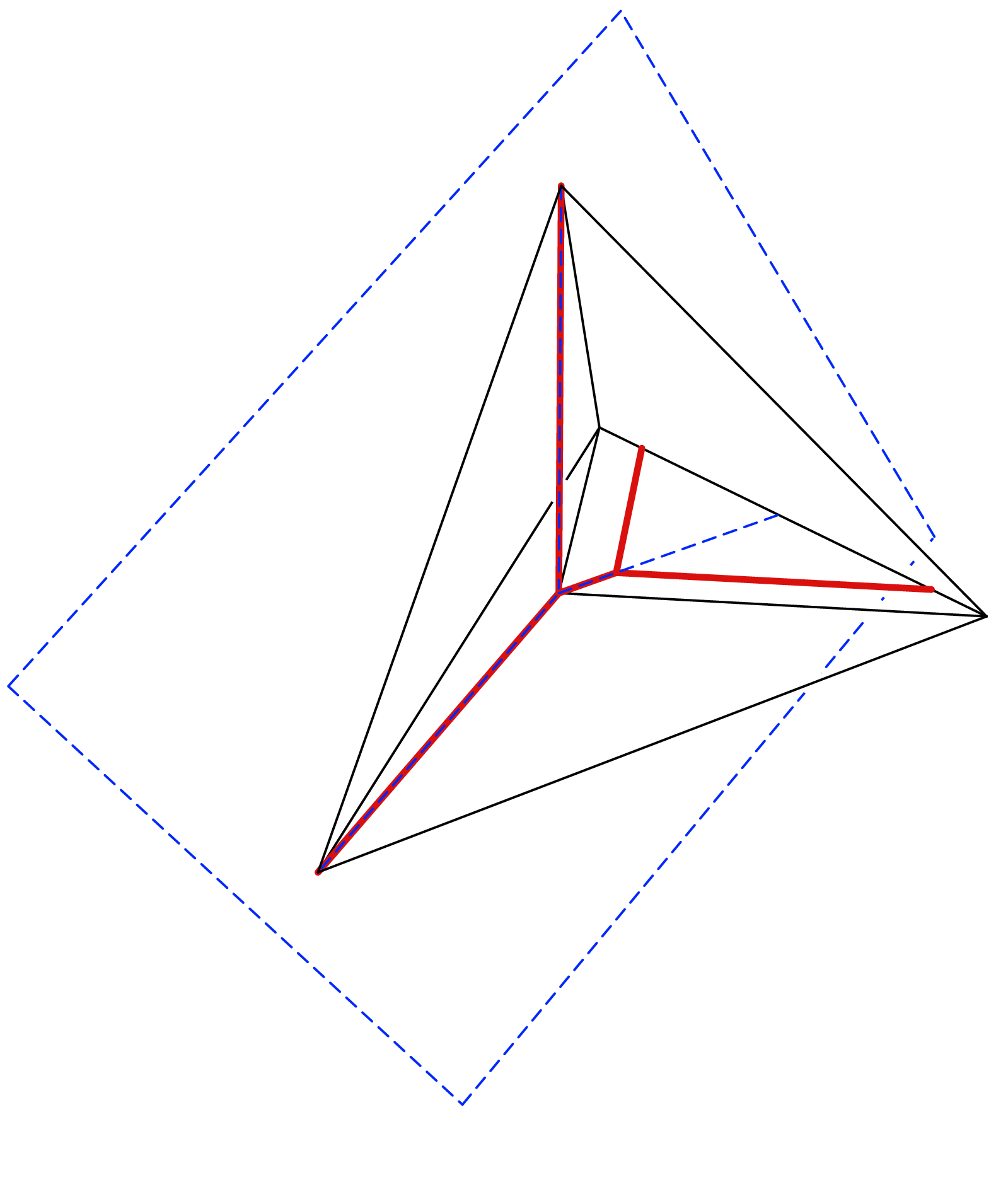}\hspace{.2in}\includegraphics[height=90mm]{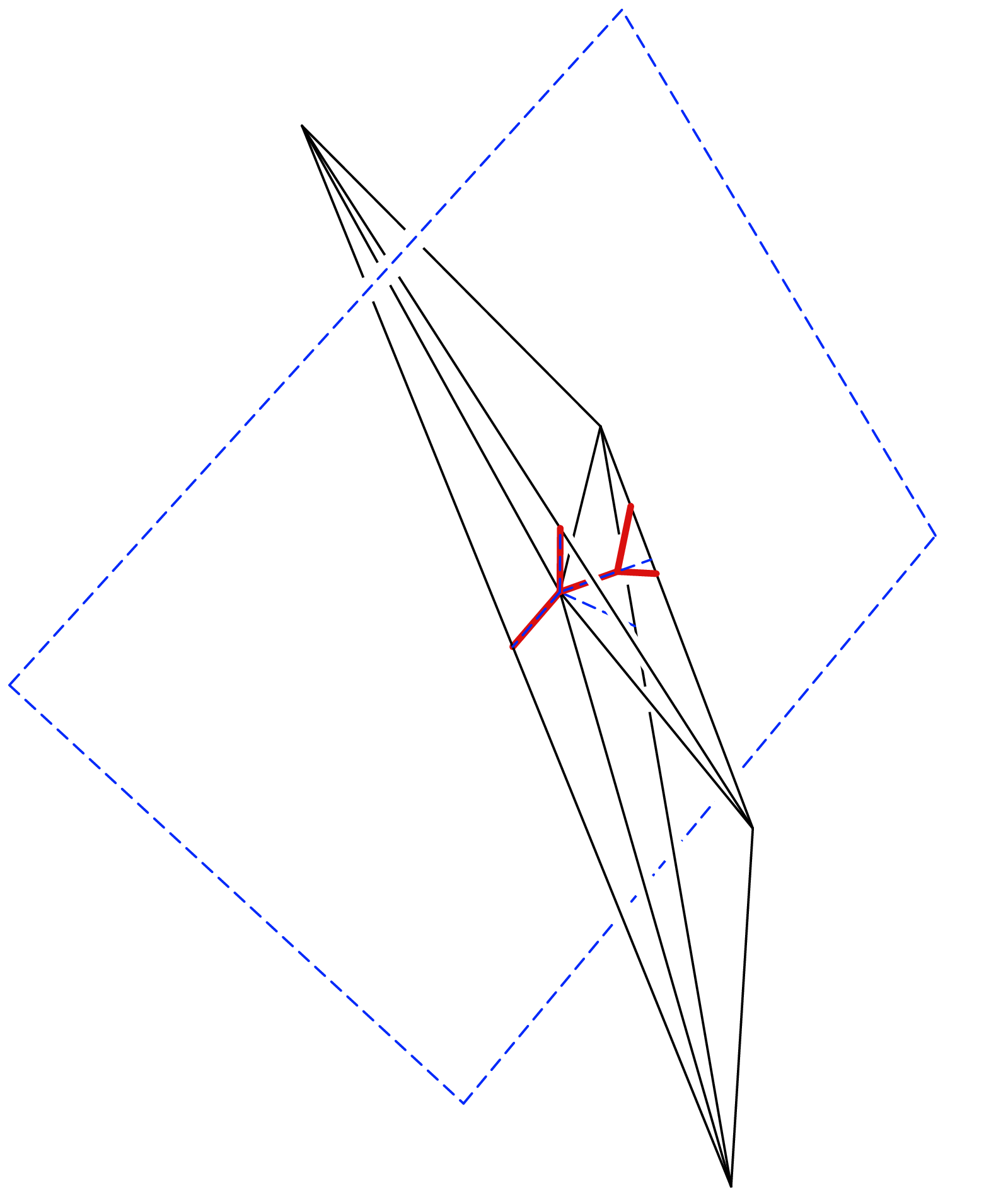}

\vspace{-.2in}
\caption{Tropical line in a tropical plane and in Vigeland's cubic surface}\label{f:lineinplane}
\end{center}
\end{figure}

In this section, we explore an example due to Vigeland \cite{Vigeland}
of some unusual tropical curves in unimodular tropical surfaces
in $\R^3$.  It is well-known
that a generic algebraic surface of degree $\delta$ contains exactly
27 lines if $\delta=3$ and 
that it contains no lines at all if $\delta>3$.  However, Vigeland
exhibited for any $\delta\geq 3$ a tropical surface in $\R^3$ of
degree $\delta$ containing an infinite family of tropical lines.  Moreover such tropical lines occur on the tropicalizations of generic surfaces.
 We show that none of the $3$-valent lines in Vigeland's family lift.  
  
 Vigeland began with a particular unimodular triangulation $S$ of the dilated
 standard tetrahedron $\delta \Delta_3$ for $\delta\geq 3$.  This
 triangulation contains the tetrahedron 
 \[\Omega_\delta=\Conv(\{(0,0,0),(0,0,1),(\delta-1,1,0),(1,0,\delta-1)\}).\]
Moreover, $S$ is coherent in that it is induced by a degree $\delta$ polynomial $p_1\in\K[x_1^\pm,x_2^\pm,x_3^\pm]$.   By possibly making a change of variables of the form $x_i\mapsto t^{w_i}x_i$, we may suppose that $0$ is the vertex of  $\Trop(V(p_1))$ dual to $\Omega_\delta$.  Then $\init{0}{p_1}$ is a polynomial of the form
\[f=A+Bx_3+Cx_1^{\delta-1}x_2+Dx_1x_3^{\delta-1}\] 
with $A,B,C,D\in\k^*$. The tropicalization $\Trop(V(f))$ is the positive codimension skeleton of the 
normal fan to $\Omega_\delta$.  It is the image of the standard tropical plane under a monomial change of variables.
For $a\in\R_{> 0}$, consider the tropical curve $L_a$ in $\R^3$ 
given by
\begin{eqnarray*}
L_a&=&\{re_3|r\geq 0\}\cup \{r(-e_1-e_2-e_3)|r \geq 0\}\cup \{r(e_1+e_2)|0\leq r\leq a\}\\
&&\cup \{a(e_1+e_2)+re_1|r\geq 0\}\cup \{a(e_1+e_2)+re_2|r\geq 0\}.
\end{eqnarray*}
Vigeland verifies that such a curve lies on $\Trop(V(f))$ and, in fact, lies on $\Trop(V(p_1))$.  We claim that $\Star_0(L_a)$ does not lift to a classical curve of $V(f)$.  This, in turn, shows that $L_a$ does not lift to a classical curve on $V(p_1)$.

By way of comparison, we note that $L_a$ does  lift to a curve on a classical plane.  The tropical curve $L_a$ is the tropicalization of the classical line parameterized by $s\mapsto (s,s+t^a,s+1)$ which lies on the plane $V(-2x+y+z-(1+t^a))$.  By taking tropicalizations, we see $L_a$ lies on the standard tropical plane.  

The star of $L_a$ at the origin spans a classical plane $U$.
The tropical line $L_a$ in the standard tropical plane and in $\Trop(V(f))$ for $\delta=3$ is shown in red in Figure \ref{f:lineinplane}.  The plane $U$ and its intersection with each surface is depicted in blue.  
The inclusion $i: U \hookrightarrow \R^3$ is induced by a map of lattices 
\[f_1\mapsto e_3, f_2\mapsto e_1+e_2\]
where $f_1,f_2$ form a basis for the lattice in $U$.  This induces a
projection $i^\wedge$ of dual lattices.  Under this projection, the
vertices of $\Omega_\delta$ are sent to 
\[\mathcal{B}=\{(0,0),(1,0),(0,\delta),(\delta-1,1)\}.\]
For the case of $\delta=3$, this is Example \ref{ex:fourterm}. By Lemma \ref{lemma:properint}, $U$ meets $\Trop(V(f))$ properly.
Consequently, the stable intersection of $U$ and $\Trop(V(f))$ is isomorphic to $\Trop(V(i^*f))$, the
positive codimension skeleton of the normal fan to
$\Conv(\mathcal{B})$.   This is depicted  for $\delta=3$ on the left in
Figure \ref{f:fourterm}.  Note that because $\Conv(\mathcal{B})$ has no pairs of parallel edges, $\Trop(V(i^*f))$  does not contain a classical
line.  However, $\Star_0(L_a)$ is a tropical line in $U$ which is
different from the stable intersection.  This shows, by Proposition
\ref{prop:onKvertex}, that $L_a$ does not lift to a curve in
any surface whose tropicalization equals $\Trop(V(p_1))$. 

\excise{The tropical line in the standard tropical plane and in $\Trop(V(f))$ for $\delta=3$ is shown in red in Figure \ref{f:lineinplane}.  The plane $U$ and its intersection with each surface is depicted in blue.  Note that the stable intersection of $U$ with the star of $\Trop(V(f))$ at the origin does not contain a classical line and that the stable intersection of $\Trop(V(f))$ with $U$ is not locally equivalent to $\Star_0(L_a)$.}

\section{Four-Term Polynomials} \label{sec:fourterms}

In this section, we classify reducible four-term Laurent polynomials
in two variables satisfying certain conditions to complete the proof of
Proposition \ref{prop:starsin3space}.  Let $\k$ be an algebraically closed field of characteristic $0$.
We will make use of Bernstein's theorem \cite{Bernstein} for $n=2$ in the following form \cite[Prop 1.2]{HS}:

\begin{theorem} \label{t:bernstein} Let $f_1,\dots,f_n\in\k[x_1^\pm,\dots,x_n^\pm]$ be Laurent polynomials.  Suppose 
that for any $w\in\Q^n\setminus\{0\}$, the initial forms $\init{w}{f_1},\dots,\init{w}{f_n}$ have no 
common zero in $(\k^*)^n$.  Then the number of common zeroes (counted with multiplicity) of
$f_1,\dots,f_n$ in $(\k^*)^n$ is 
the mixed volume of the Newton polytopes of $f_1,\dots,f_n$.
\end{theorem}

We have the  lemma below which will be useful for verifying the conditions of Bernstein's theorem.  For $\ca\subset\Z^n$, let $\Face_w(\ca)$ be the set of points  in $\ca$ along which $\U\cdot w$ is minimized.

\begin{lemma} \label{lem:intersectionintorus}  Let $f=gh$ be a polynomial with support set $\ca\subset\Z^2$.  Let $w\in\Q^2\setminus\{0\}$ be such that $|\Face_w(\ca)|\leq 2$.  Then
$\init{w}{g}$ and $\init{w}{h}$ do not have a common zero in $(\k^*)^2$.
\end{lemma} 

\begin{proof}
The Laurent polynomial $\init{w}{f} = (\init{w}{g})(\init{w}{h})$ is
the sum of monomials of $f$ along $\Face_w(P(f))$, which is either a
vertex or an edge of $P(f)$.  
Then $\init{w}{f}$ is a monomial or binomial, which in either case has
no repeated roots in $(\k^*)^2$. 
Therefore, $\init{w}{g}$ and $\init{w}{h}$ have no common
zeroes in $(\k^*)^2$. 
\end{proof}

\begin{definition} A binomial $b \in\k[x_1^\pm,\dots,x_n^\pm]$ is said to be \emph{minimal} if 
\[b=cx_1^{a_1}\dots x_n^{a_n}+dx_1^{b_1}\dots x^{b_n}\]
where $\gcd(a_1-b_1,\dots,a_n-b_n)=1$.
\end{definition}

Note that $b$ is a minimal binomial if and only if the Newton polytope $P(b)$ is a segment of lattice length
one.  The tropicalization of $V(b)$, $\Trop(V(b))$ is the classical hyperplane
in $\R^n$ of weight $1$ through the origin that is orthogonal to the segment $P(b)$.  

\begin{proposition} \label{prop:binomialfactor}
Let $f\in\k[x_1^\pm,x_2^\pm]$ be a polynomial with four monomials such
that $P(f)$ is a quadrilateral.  Suppose that the $\Z$-affine span of the support of $f$ is $\Z^2$.  If $f=gh$ is a non-trivial factorization, then one of the factors is a minimal binomial.  Consequently, $\Trop(V(f))$ contains a classical line of weight $1$ as a tropical cycle summand.
\end{proposition}

Our strategy to prove Proposition~\ref{prop:binomialfactor} is to show that the only possible singularity of $V(f)$ is a single node and therefore that $V(g)$ and $V(h)$ intersect in exactly one point.  By the use of 
Bernstein's theorem, we can constrain the Newton polygons of $g$ and $h$.  This, in turn, will restrict the support set of $f$.

\begin{lemma} $V(f)$ has at most one singular point in $(\k^*)^2$.
\end{lemma}

\begin{proof}
Suppose $f$ is singular.  By multiplying $f$ by a monomial, we may suppose that it has the form
\[f=1+d_ax_1^{a_1}x_2^{a_2}+d_bx_1^{b_1}x_2^{b_2}+d_cx_1^{c_1}x_2^{c_2}.\]
Since $P(f)$ is a quadrilateral,
$Q=\Conv\{(a_1,a_2),(b_1,b_2),(c_1,c_2)\}$ is a non-degenerate
triangle. By applying the $(\k^*)^2$-action, we may suppose that $f$
is singular at $(1,1)$.  Therefore, \[0=f(1,1)=f_{x_1}(1,1)=f_{x_2}(1,1).\]
We rewrite these equations as a linear system in $d_a,d_b,d_c$.  By a straightforward computation, the determinant of the linear system is (up to sign) equal to the area of $Q$.  Therefore, there is a unique choice of coefficients $d_a,d_b,d_c$ that gives a curve singular at $(1,1)$.

Suppose that there is an additional singularity at a point $(\xi_1,\xi_2)$.  Therefore, $f(\xi_1^{-1}x_1,\xi_2^{-1}x_2)$ is singular at $(1,1)$ and hence is equal to $f(x_1,x_2)$.
By equating coefficients, we conclude that
$\xi_1^{a_1}\xi_2^{a_2}=\xi_1^{b_1}\xi_2^{b_2}=\xi_1^{c_1}\xi_2^{c_2}=1$.
Since the $\Z$-linear span of the vertices of $Q$ is $\Z^2$, we must have $(\xi_1,\xi_2)=(1,1)$.
\end{proof}

\begin{lemma} Any singular point of $V(f)$ in $({\k}^*)^2$ is a node.
\end{lemma}

\begin{proof}
Without loss of generality, we again suppose the singular point is at $(1,1)$.
By a computation, the Hessian at $(1,1)$ is the product of the areas
of the four triangles formed by triples in $i^\wedge(\ca)$ up to a
non-zero constant.   Since this is non-zero, near $(1,1)$, $V(f)$ is analytically
isomorphic to the zero-locus of a quadratic form, and hence has a node
at that point.
\end{proof}

\begin{lemma} \label{lem:intersectionatmost1} $V(g)$ and $V(h)$ have intersection number at most $1$ in $(\k^*)^2$.
\end{lemma} 

\begin{proof}  Any point of intersection of $V(g)$ and $V(h)$ gives a singular point of $V(f)$.  There
  is only one singular point and it is a node, which corresponds to two curves meeting transversally.  In this case, the intersection multiplicity is $1$.
\end{proof}

Now there are three different cases: $P(g)$ and $P(h)$ could both be
segments, $P(g)$ could be a polygon and $P(h)$ a segment, or $P(g)$
and $P(h)$ could both be polygons. 

In the first case, $P(g)$ and $P(h)$ cannot be parallel and so $P(f)$
is a parallelogram.   By Bernstein's theorem, for $V(g)$ and $V(h)$ to have intersection number $1$, $P(g)$ and $P(h)$ must be minimal binomials whose primitive integer vectors span $\Z^2$.
In the second case, because the mixed volume of $P(g)$ and $P(h)$ is $1$, $h$ must be a minimal binomial.
We now eliminate the third case to complete the proof of Proposition \ref{prop:binomialfactor}.

\begin{lemma} $P(g)$ and $P(h)$ cannot both be polygons.
\end{lemma}

\begin{proof}
We will show that if both $P(g)$ and $P(h)$ are polygons then their
mixed area is at least $2$. Because Lemmas~\ref{lem:intersectionintorus} applies for all $w\in\Q^2\setminus\{0\}$, by Bernstein's theorem, $V(g)$ and $V(h)$ have at least $2$ common zeroes (counted with multiplicity), contradicting Lemma~\ref{lem:intersectionatmost1}.  

First suppose both $P(g)$ and $P(h)$ are quadrilaterals. Since
their Minkowski sum is also a quadrilateral, the two polygons have all
of their edge directions in common. Choose
$w \in \Q^2\setminus\{0\}$ such that $\Face_w P(g)$ is a vertex $v_g$ and
$\Face_{-w} P(h)$ is a vertex $v_h$. Draw the Minkowski sum $P(g)+P(h)$
by placing the two polygons so that $v_g$ and $v_h$
coincide. Then $P(g)+P(h)$ is the union of $P(g)$, $P(h)$, and two
lattice parallelograms, all with disjoint interiors. The area of each
parallelogram is a positive integer, so the mixed area of $P(g)$ and
$P(h)$ is at least two.

Next suppose one polygon, say $P(g)$, is a triangle and $P(h)$ is a
quadrilateral. Then all three edge directions of $P(g)$ are also edge
directions of $P(h)$. In order for these three edges of $P(h)$ not to
close up (as they do in $P(g)$), at least one edge of $P(h)$ must be
of different length than its counterpart in $P(g)$. Since both $P(g)$ and $P(h)$
are lattice polygons, either $P(g)$ or $P(h)$ has an edge $e$ of lattice length $k\geq 2$.  If $P(h)$ contains the edge $e$, draw the Minkowski sum by fixing $P(h)$ and then
placing $P(g)$ at either end of $e$. Then $P(g) + P(h)$ contains the
union of $P(g)$, $P(h)$, and a lattice parallelogram with $e$ as one
of its edges, all with disjoint interiors. So the mixed area is at
least $k$.  The case that $P(g)$ contains $e$ is identical.

Finally, suppose $P(g)$ and $P(h)$ are both triangles. Then they share
two edge directions. But since the third edge direction is \emph{not}
shared, one of the common edges must be longer in one polygon, say
$P(h)$, then in the other. So again $P(h)$ has an edge of lattice
length $k \geq 2$, and we can proceed as in the previous case.   
\end{proof}

\section{Lifting Curves in Tropical Planes in High Dimensional Space} \label{sec:maclagan}
Using an idea of Gibney-Maclagan
\cite[Sec. 4.1]{GM} who consider the constant coefficient case, we give another example showing  that the 
relative lifting problem is not combinatorial.  In other words, we
will exhibit two surfaces $S^\circ,{S'}^\circ\subset(\K^*)^n$ with
$\Trop(S^\circ)=\Trop({S'}^\circ)$ and a tropical curve
$\Gamma\subset\Trop(S^\circ)$ such that $\Gamma$ lifts in $S^\circ$
but not in ${S'}^\circ$. 

Let $d$ be a positive integer.  Let $x,y$ be the coordinates on $(\K^*)^2$.
Pick a set $\mathring{T}$ of $\binom{d+2}{2}-1$ points in tropical
general position in $\R^2$ \cite[Def 4.7]{Mi03}.  Let $\mathring{P}$
be a set of points  in $(\K^*)^2$ that lifts $\mathring{T}$.  There is
a unique curve $C$ of degree $d$ in $\P^2$ passing through
$\mathring{P}$.  Pick a generic point $p\in C\cap (\K^*)^2$ such that
$v(p)$ does not share an $x$- or $y$-coordinate with any of the points in $\mathring{T}$.
Let $P=\mathring{P}\cup\{p\}$.
Let $T=\mathring{T}\cup \{v(p)\}$.  We may suppose by general position considerations that no line of the form $x=x_k$ or $y=y_k$ through a point of $P$ is tangent to $C$.

Let $P'$ be a set of $N=\binom{d+2}{2}$ points of $(\K^*)^2$ in general position lifting $T$.  It is possible to choose such a set since the set of such lifts is Zariski dense in $(\P^2)^N$ by the arguments of \cite{PFibers}.
Note that there is no curve of degree $d$ passing through $P'$.
Enumerate the points of $P$ and $P'$ as $\{(x_1,y_1),\dots,(x_N,y_N)\}$ and $\{(x'_1,y'_1),\dots,(x'_N,y'_N)\}$, 
respectively.  Define rational functions $l_1,\dots,l_N,l'_1,\dots,l'_N$ by 
\[l_k=\frac{y-y_k}{x-x_k},\ l'_k=\frac{y-y'_k}{x-x'_k}.\]
Let $S,S'\subset(\P^2\times (\P^1)^N)$ be the closures of the graphs of $j=(l_1,\dots,l_N)$ and $j'=(l'_1,\dots,l'_N)$.    Note that $S$ and $S'$ are the blow-ups of $\P^2$ at $P$ and $P'$ since blowing up those points resolves the indeterminacy of the rational functions. 
Let $S^\circ=S\cap \left( (\K^*)^2\times(\K^*)^N\right)$,
${S'}^\circ=S'\cap\left( (\K^*)^{2}\times(\K^*)^N\right)$.

\begin{lemma} $\Trop(S^\circ)=\Trop({S'}^\circ)$
\end{lemma}

\begin{proof}
Let $z_1,\dots,z_N$ be the coordinates on $(\K^*)^N$.  For $k\in\{1,\dots,N\}$, let $f_k$ be the polynomial
\[f_k=(x-x_k)z_k-(y-y_k)\]
in variables $x,y,z_1,\dots,z_N$.
This defines a surface $V(f_k)\subset (\K^*)^2\times(\K^*)^N$.  The image of $V(f_k)$ under projection to the copy of $(\K^*)^3$ given by coordinates $x,y,z_k$ is the graph of $l_k$.  Let $J$ be the intersection,
\[J=\bigcap_{k=1}^N \Trop(V(f_k)).\]
We will show $\Trop(S^\circ)=J$.
An identical argument will apply to $\Trop({S'}^\circ)$.
Now, by Kapranov's theorem $\Trop(V(f_k))$ depends only on  $v(x_k)$ and $v(y_k)$.
Since $v(x_k)=v(x'_k)$ and $v(y_k)=v(y'_k)$, we can conclude $\Trop(S^\circ)=\Trop({S'}^\circ)$.

Clearly $\Trop(S^\circ)\subseteq J$.
We now show $J\subseteq\Trop(S^\circ)$.  The Newton polytope of
$f_k$ is the unimodular tetrahedron 
\[P(f_k) = \Conv\{(0,0,0),(0,1,0),(0,0,1),(1,0,1)\} \]
in the space $\R^3$ with coordinates $(x,y,z_k)$. Note that $\Trop(V(f_k))$ has $6$ top-dimensional cells.  Four of them are cut out by 
one of the following equalities:
\begin{enumerate}
\item $v(z_k)=v(y)-v(x)$
\item $v(z_k)=v(y)-v(x_k)$
\item $v(z_k)=v(y_k)-v(x)$
\item $v(z_k)=v(y_k)-v(x_k)$
\end{enumerate}
together with two inequalities.
The other two are cut out by $v(x)=v(x_k)$ or $v(y)=v(y_k)$ together with some inequalities involving $v(z_k)$.    We call the last two cells {\em vertical} as they have lower-dimensional image under the projection $\R^2\times\R^N\rightarrow\R^2$.  
All points in any higher-codimensional cell must satisfy either $v(x)=v(x_k)$ or $v(y)=v(y_k)$ and will be vertical.
A point in $\R^2\times\R^n$ is in the relative interior of at most two vertical cells. 

By taking the common refinement of the fans $\Trop(V(f_k))$ for
all $k$, we induce a polyhedral structure on their intersection, $J$.  We claim that the
top-dimensional cells of $J$ are two-dimensional.  We know that they are at least two-dimensional.  Let $q$ be a point
of $J$ lying in the relative interior of a cell $F$.  If $q$ lies in
no vertical cell, then  $v(x)$ and $v(y)$ determine $v(z_k)$ for all
$k$.  In that case, $F$ is at most two-dimensional.  If $q$ lies in a
vertical cell in $\Trop(V(f_l))$ for exactly one $l$ then either $v(x)=v(x_l)$ or $v(y)=v(y_l)$.
Therefore, $v(z_l)$ and one of $\{v(x),v(y)\}$ can vary.  These variables determine $v(z_k)$ for $k\neq l$.  Then $F$ is again two-dimensional.  The case where $q$ is in vertical cells of $\Trop(V(f_k))$ and $\Trop(V(f_l))$ imposing $v(x)=v(x_k)$ and $v(y)=v(y_l)$ is similar.

By this argument, any two-dimensional cell $F$ of $J$ lies in
the interior of top-dimensional cells of $\Trop(V(f_k))$ for all $k$.
If $F$ lies in no vertical cell, then the affine spans of the cells of
$\Trop(V(f_k))$ containing it are of the form $v(z_k)=h_k(v(x),v(y))$
where $h_k$ is a linear (possibly constant) function.  
These affine spans intersect transversely.  Therefore,
$F\subset\Trop(S^\circ)$ by the transverse intersection lemma.
If $F$ lies in one vertical cell, 
then for some $1 \leq l \leq N$, the affine spans of the cells
containing it are cut out by equations of the form
$v(z_k)=h_k(v(x),v(y))$ for all $k \neq l$ and $v(x)=v(x_l)$ or
$v(y)=v(y_l)$.  These affine spans also intersect transversely, so
$F\subset\Trop(S^\circ)$.  The same argument applies when $F$ is contained in two vertical cells and shows $F\subset\Trop(S^\circ)$.
  Since every two-dimensional cell of $J$ is
contained in $\Trop(S^\circ)$ and since $\Trop(S^\circ)$ is closed and purely two-dimensional, $J=\Trop(S^\circ)$.
\end{proof}

Let $\Gamma=\Trop(j(C))\subset\Trop(S^\circ)=\Trop({S'}^\circ)$.  We will show that $\Gamma$ does not lift to a curve in ${S'}^\circ$.
Let $p_k:\P^2\times (\P^1)^N\rightarrow \P^1$ be the projection onto the $k$th factor in $(\P^1)^N$.  

\begin{lemma} For all $k$, $\deg(p_k:C\rightarrow \P^1)=d-1$.
\end{lemma}

\begin{proof}
Let $E_1,\dots,E_N$ be the exceptional divisors in $S$ and $H$ be the class of a line pulled back from $\P^2$.  The class of $C$ is $dH-\sum_n E_n$.   
For $q\in \P^1$ chosen generically, $p_k^{-1}(q)\cap S$ is a curve in the class $H-E_k$.  Intersecting it with $C$, we get $d-1$.
\end{proof}

\begin{proposition} $\Gamma$ does not lift to any curve on ${S'}^\circ$.
\end{proposition}

\begin{proof}
Suppose it lifts to a curve $C'$.  By considering the projections $p:S^\circ\rightarrow(\K^*)^2$, $p':S'^\circ\rightarrow(\K^*)^2$, we see that $\Trop(p'(C'))=p_*(\Gamma)=\Trop(p(C))$ and that $p'(C')$ is of degree $d$.

Now, by \cite[Thm 1.1]{StuTevelev}, the degree of $p_k:C\rightarrow \P^1$ is equal to the tropical degree of the map $p_{k*}:\Trop(C)\rightarrow \R^1$.  But this is the same map as $p'_{k*}:\Trop(C')\rightarrow \R^1$ and therefore has the same tropical degree.  This, in turn, is equal to the degree of $p'_k:C'\rightarrow \P^1$.  

Write the class of $C'$ as $dH-\sum a_iE'_i$ for $a_i\in\Z$.  By
intersecting with ${p'}_k^{-1}(q)$ which is of class $H-E'_k$, we see
that 
\[d-1=(dH-\sum a_iE'_i)\cdot(H-E'_k)=d-a_k\]
and hence $a_k=1$.  But since the points $(x'_k,y'_k)$ were chosen in
general position, there is no curve of degree $d$ through all of these
points and hence no curve of class $dH-\sum E'_i$.
\end{proof}

\bibliographystyle{plain}

\end{document}